\documentclass[12pt,reqno,a4paper]{amsart}

\usepackage{mathtools,color,bbold}

\usepackage[colorlinks=true, linkcolor=blue, urlcolor=blue, citecolor=blue, pagebackref=true]{hyperref}

\usepackage[abbrev,bibtex-style]{amsrefs}

\usepackage{mathtools}
\usepackage{ amssymb }

\usepackage{amsfonts}

\numberwithin{equation}{section} \theoremstyle{plain}
\newtheorem{theorem}{Theorem}
\newtheorem{proposition}[theorem]{Proposition}
\newtheorem{lemma}[theorem]{Lemma}
\newtheorem{corollary}[theorem]{Corollary}

\theoremstyle{definition}

\renewcommand{\leq}{\leqslant}
\renewcommand{\geq}{\geqslant}

\newcommand\E{\mathbb{E}}

\newcommand\Z{\mathbb{Z}}
\newcommand\R{\mathbb{R}}

\newcommand\C{\mathbb{C}}
\newcommand\N{\mathbb{N}}

\newcommand\F{\mathbb{F}}
\newcommand\Q{\mathbb{Q}}

\newcommand\Supp{\operatorname{Supp}}

\renewcommand\P{\mathbb{P}}

\renewcommand\Re{\operatorname{Re}}

\renewcommand{\a}{\alpha}
\renewcommand{\b}{\beta}
\renewcommand{\d}{\delta}

\newcommand{\e}{\varepsilon}

\newcommand{\wt}{\widetilde}

\newcommand\eps{\varepsilon}

\begin{document}

\title[Cut-off phenomenon for ax+b]{Cut-off phenomenon for the ax+b Markov chain over a finite field}

\author{Emmanuel Breuillard}
\address{Centre for Mathematical Sciences\\
Wilberforce Road\\
Cambridge CB3 0WB\\
UK }
\email{emmanuel.breuillard@maths.cam.ac.uk}

\author{P\'eter P. Varj\'u}
\address{Centre for Mathematical Sciences\\
Wilberforce Road\\
Cambridge CB3 0WB\\
UK }
\email{pv270@dpmms.cam.ac.uk}

\thanks{EB has received funding from the European Research Council (ERC) under the European Union's Horizon 2020 research and innovation programme (grant agreement No. 617129);
PV has received funding from the Royal Society and the European Research Council (ERC) under the European Union's Horizon 2020 research and innovation programme (grant agreement No. 803711).}


\subjclass[2010]{}

\begin{abstract}We study the Markov chain $x_{n+1}=ax_n+b_n$ on a finite field $\F_p$, where $a \in \F_p^{\times}$ is fixed and $b_n$ are independent and identically distributed random variables in $\F_p$. Conditionally on the Riemann hypothesis for all Dedekind zeta functions, we show that the chain exhibits a cut-off phenomenon for most primes $p$ and most values of $a \in \F_p^\times$. We also obtain weaker, but unconditional, upper bounds for the mixing time.
\end{abstract}

\maketitle

\section{Introduction}

Let $p$ be a prime number and $\F_p$ the field with $p$ elements. Consider the Markov chain on $\F_p$
\begin{equation}\label{chain}x_{n+1}=ax_n+b_n,\end{equation}
where the multiplier $a$ is non-zero and the $b_n$'s are independent random variables taking values in $\F_p$ with a common law $\mu$. We will assume that the support of $\mu$ has at least two elements, a condition which is easily seen to be equivalent to the aperiodicity of the Markov chain. 
The chain will then mix and eventually equidistribute towards a stationary measure, which must clearly be the uniform distribution $u$ on $\F_p$. For $\delta \in (0,1)$, it is customary to define the mixing time $T(\delta) \in \N$ as
$$T(\delta) := \inf\{ n \in \N | \|\mu_a^{(n)} - u\|_{TV} \leq \delta\}.$$
Here $\mu_a^{(n)}$ denotes the law of $x_n$ on $\F_p$ (starting from $x_0=0$ say) and the norm is the total variation distance:
$$\|\pi_1 - \pi_2\|_{TV} = \sup_{A \subset \F_p} |\pi_1(A)- \pi_2(A)|=\frac{1}{2}\|\pi_1-\pi_2\|_1 \in [0,1]$$
for any two probability measures $\pi_1,\pi_2$ on $\F_p$.

In this paper, we will be interested in giving good bounds on $T(\delta)$ in terms of $p$.  When $a=1$ or when $a$ has small multiplicative order in $\F_p^{\times}$, then the Markov chain exhibits a diffusive behaviour and $T(\delta)$ will typically be of order $p^2$. However, if $a$ is more typical in $\F_p^{\times}$, we expect the mixing time to be of order $\log p$, which is of course optimal up to a multiplicative constant when the size of the support $\Supp(\mu)$ of $\mu$ is bounded independently of $p$, because only at most $|\Supp(\mu)|^n$ sites can be visited by the Markov chain after $n$ steps (see $(\ref{ent-b})$ below).

This problem, inspired by the pseudo-random generators used in the first electronic computers and based on the idea, due to D.H. Lehmer \cite{lehmer-comp}, of iterating the map $x \mapsto ax+1$ on $\F_p$, has been studied in a beautiful paper of Chung-Diaconis-Graham \cite{CDG-RW} in the special case when $a=2 \mod p$ for all $p$ and $b_n$ is uniform in $\{-1,0,1\}$. Their work was extended by Hildebrand \cites{Hil-thesis, Hildebrand-AnnalsOfProba} to the case when $a$ is allowed to be random, i.e. $a = \lambda_n \mod p$ and $b_n=\beta_n \mod p$, where the $\lambda_n$'s and $\beta_n$'s are two independent sequences of i.i.d. integer valued random variables, each chosen according to some fixed probability measure supported on $\Z$. In this case, under mild assumptions of the measures,  mixing occurs in $O(\log p \log \log p)$ steps, and this bound in sharp in some cases (e.g. $a=2$, $\mu$ uniform on $\{-1,0,1\}$ and $p$ is a  Mersenne prime, i.e.  of the form $p=2^n-1$). See  \cite{CDG-RW}*{Theorem 1} as well as \cite{hildebrand-on-the-CDG-process}.  Still, the sharpness seems to be occuring only for a thin, density zero set of primes (recall that Mersenne primes have density zero and  form a conjecturally infinite set of primes).  This  hints at the possibility that an upper bound of the form $O(\log p)$ might hold for almost all primes and indeed Chung-Diaconis-Graham prove in \cite{CDG-RW} the analogous statement in $\Z/p\Z$, where $p$ is an arbitrary integer not required to be prime.

We do not answer this specific question, but rather study what happens when $a$ is an arbitrary element of $\F_p$ (as opposed to being the mod $p$ quotient of a fixed value $a\in \Z$ independent of $p$ as in the above mentioned works). It was observed by Bukh, Harper, Helfgott and Lindenstrauss (see \cite{Hel-Arizona}*{Exercises 3.13 and 3.14}), that a general polylogarithmic upper bound on the mixing time, valid for all primes $p$ and for all $a$ of sufficiently large multiplicative order, can be deduced easily from some estimates of Konyagin \cite{Kon-RW}. This gives 
\begin{equation}\label{kon}T(\delta) = O((\log p)^{2+o(1)})\end{equation} provided the multiplicative order of $a$ is at least $O(\log p\log\log p)$ in $\F_p$. See Section \ref{konyagin} below, where a more precise statement is given. Improving on this bound would touch upon some delicate number theoretical issues: for example proving a $O(\log p)$ upper bound for most primes $p$ and for all values of $a$ of sufficiently large multiplicative order would already imply the celebrated Lehmer conjecture (see Section \ref{lehmer-sec} and \cite{breuillard-varju-Lehmer}).

In this paper, we prove three results valid for most values of the multiplier $a$. First we prove that the $O(\log p\log \log p)$ upper bound holds for all primes (Theorem \ref{loglog}). Then we will show a stronger logarithmic $O(\log p)$ upper bound, which is only valid for most (in a certain weak sense) primes $p$ (Theorem \ref{uncond}). Finally, conditionally on the Riemann hypothesis for Dedekind zeta functions of number fields, we will establish that  the Markov chain exhibits a cut-off phenomenon for most (this time density $1$) primes $p$ (Theorem \ref{TVcutoff}). A sequence of Markov chains is said to exhibit the cut-off phenomenon if there is a sharp phase transition to equidistribution, that is we have $T(\d_1)/T(\d_2)\to 1$ as we go through the sequence of chains for any $\d_1,\d_2\in(0,1)$.  We refer the reader to   \cites{diaconis-survey, peres-levin} and references therein for an introduction to cut-off phenomena for Markov chains.

\begin{theorem}\label{loglog}  Let $p$ be a prime number and $\mu$ a probability measure on $\F_p$ supported on at least two elements. Let $\eps \in (0,1)$. Then for all but an $\eps$-fraction of all $a \in \F_p^{\times}$, the mixing times of the Markov chain $(\ref{chain})$ satisfy for any $\delta \in (0,\frac{1}{2})$:
$$T(\delta) \leq T_2(\delta) \leq C_{\eps,\mu}\log p (\log\log p +\log (C_{\eps,\mu}  \delta^{-1})),$$
where $C_{\eps,\mu}=O(\eps^{-1}H_2(\mu)^{-1})$ for some absolute implied constant.
\end{theorem}

Here $T_2(\delta):=\inf\{ n \in \N | p\|\mu_a^{(n)} - u\|^2_{2} \leq \delta^2\}$ is the $\ell^2$-mixing time, and $H_2(\mu):= - \log  \|\mu\|_2^2$  is the R\'enyi entropy of $\mu$ of order $2$, where $\|\mu\|_2^2= \sum_{x \in \F_p} \mu(x)^2$. 

 We recall that for all $\delta>0$ we have $T(\delta)\leq T_2(2\delta) \leq T_2(\delta)$, which  is a direct consequence of the Cauchy-Schwarz inequality:
$$\|\mu_a^{(n)} - u\|^2_{1}  \leq  p\|\mu_a^{(n)} - u\|^2_{2}.$$

In the next results we will improve the above bound for a typical prime. For this reason we need to restrict somewhat the class of chains of the form $(\ref{chain})$ that we consider by assuming that the $b_i$'s are the mod $p$ values of a fixed sequence of integer valued i.i.d. random variables whose law $\mu$  is supported on $\Z$ and independent of $p$. 

\begin{theorem}\label{uncond}Fix a probability measure $\mu$ on $\Z$ supported on a finite set of at least two elements, and consider the Markov chain $(\ref{chain})$. There is a possibly empty exceptional set of primes $B$ such that for every $X\ge1$,
\begin{equation}\label{thinn}\sum_{p \in B \cap [e^{X/2},e^X]} \frac{1}{p} \leq \frac{(\log X)^6}{\sqrt{X}}\end{equation}
with the property that given $\eps,\delta \in (0,1)$, if $p$ is a prime larger than a constant depending only on $\mu$ and these parameters, and $p \notin B$, then for all but an $\eps$-fraction of all $a \in \F_p^\times$ we have:
$$T(\delta) \leq T_2(\delta) \leq \frac{5}{H_2(\mu)} \log p,$$ where  $H_2(\mu):= - \log  \sum_{z \in \Z} \mu(z)^2$. 
\end{theorem}

Note that condition $(\ref{thinn})$ implies that $\sum_{p \in B} p^{-1}$ converges, while it is well-known that  $\sum_{p} p^{-1}$ diverges and that indeed $\sum_{p \in [e^{X/2},e^X]} p^{-1}$ lies between two positive constants independent of $X>0$. This is the sense in which the complement of $B$ contains most primes.

Finally we state our main result establishing a cut-off under GRH. We define the entropy of $\mu$ to be $$H(\mu)=\sum_{z \in \Z} \mu(z)\log \frac{1}{\mu(z)}.$$ We also set $C_\mu=\max\{1,\log (\sup \mu/ \inf\mu)\},$ where $\sup$ and $\inf$ are taken over the support of $\mu$.

\begin{theorem}\label{TVcutoff} Assume that the Riemann hypothesis holds for the Dedekind zeta function of an arbitrary number field. Fix a probability measure $\mu$ on $\Z$ supported on a finite set of at least two elements, and consider the Markov chain $(\ref{chain})$. Fix $\eps,\kappa,\theta>0$. For $X>0$, let $n$ be the integer part of $(X+\theta \sqrt{X})/H(\mu)$. Then for all but an $\eps$-fraction of all primes $p$ in $[e^X,(1+\kappa)e^X]$ and all but an $\eps$-fraction of all $a \in \F_p^\times$ we have
$$\|\mu_a^{(n)} -u \|_{TV} \leq 5\eps^{-1}e^{-\frac{H(\mu)}{16C_\mu^2}\theta^2},$$ provided $X \ge 2\log \eps^{-1}+C(\theta^2 \log \theta)^2 + C$, where $C$ is a constant depending only on $\kappa$ and $\mu$. 
\end{theorem}

On the other hand, there is a universal lower bound on $\|\mu_a^{(n)} -u \|_{TV}$ in terms of the entropy of the walk. Mixing cannot occur before the entropy of the walk $H(\mu_a^{(n)})\leq nH(\mu)$ approaches the maximal entropy, namely $H(u)=\log p$.  In fact, we will show the following in Section \ref{total}. Let $p$ be prime, $a\in \F_p$ and $\theta\in(0,\sqrt{\log p}/2)$. If $n\ge 0$ is the integer part of $(\log p - \theta \sqrt{\log p})/H(\mu)$, we have
\begin{equation}\label{ent-b}
\|\mu_a^{(n)} -u \|_{TV} \geq 1- e^{-\theta^2} - 2e^{-\frac{H(\mu)}{C_\mu^2}\theta^2}.
\end{equation}
Combined with Theorem \ref{TVcutoff}, this yields the following.

\begin{corollary} \label{cor4} Under the assumptions of Theorem \ref{TVcutoff} the Markov chain $(\ref{chain})$ exhibits a cut-off phenomenon in total variation at $\frac{1}{H(\mu)}\log p$ for most primes $p$ and most multipliers $a \in \F_p^\times$. More precisely, for any given $\eps \in (0,1)$ there is an exceptional set of primes $E$ of density at most $\eps$, such  that for all $p\notin E$  and for all but an $\eps$-fraction of all $a \in \F_p^\times$ we have:
$$|T(\delta) - \frac{1}{H(\mu)}\log p| \leq C_{\mu,\eps} \sqrt{\log p},$$
for all $\delta \in (\eps,1-\eps)$, where $C_{\mu,\eps}=O_\mu(\sqrt{\log(1/\eps)})$.
\end{corollary}

We can also let $\eps$ tend to zero and get a set of primes with density $1$ for which a cut-off takes place, at the cost of loosening a bit the control on the cut-off window. See Corollary \ref{cor5}.

We will also prove an $\ell^2$ version of these results. The $\ell^2$ cut-off window is even narrower (of bounded size).  In fact Theorem \ref{TVcutoff} will be proven after we handle the analogous $\ell^2$ version by a small variation of that argument, see Section \ref{ell2}.

\subsection{Motivation} Our motivation for studying the Markov chain $(\ref{chain})$ is manifold. First of all this chain is the natural ``mod $p$ analogue'' of the classical Bernoulli convolutions we studied in \cite{breuillard-varju0}, where a close relationship between the entropy of the random walk and the Mahler measure of the multiplier was established. This enabled us to give a ``mod $p$'' reformulation of the Lehmer conjecture in \cite{breuillard-varju-Lehmer}. On the other hand the chain $(\ref{chain})$ is intimately related to random polynomials of large degree and is a crucial tool in our recent paper on irreducibility of random polynomials \cite{breuillard-varju-irred}, where Konyagin's estimate $(\ref{kon})$ is used to establish irreducibility under GRH\footnote{Here and everywhere in this paper, when we say that we assume GRH, we mean that we assume that the Dedekind zeta function of an arbitrary number field satisfies the Riemann hypothesis, namely has its non trivial zeroes on the critical line.}. The latter paper and the current one can be seen as companion papers, since the main irreducibility result from \cite{breuillard-varju-irred} will be the key ingredient in the proof of the cut-off in Theorem \ref{TVcutoff}. Finally the chain $(\ref{chain})$, which can also be described in terms of random walks on the affine group of the line,  can be seen as a toy model for many more sophisticated random walks on groups or homogeneous spaces of algebraic origin. 

\subsection{Strategy} We now briefly describe the main strategy behind the proof of the above results. It starts with the observation that the chain $(\ref{chain})$ at time $n$ is given by the value at $a \in \F_p^\times$ of a random polynomial \begin{equation}\label{rpo}P(x)=b_0x^{n-1}+ \ldots + b_{n-2}x+b_{n-1}\end{equation} where the $b_i \in \Z$'s are as before i.i.d.  
random variables with law $\mu$. In particular the law of $P(a)$ is exactly $\mu_a^{(n)}$. Furthermore, we can write

$$\|\mu_a^{(n)}\|_2^2 = \sum_{y \in \F_p} \mu_a^{(n)}(y)^2=  \sum_{y \in \F_p} \P(P_1(a)=P_2(a)=y)=\P(P_1(a)=P_2(a))$$
where $P_1$ and $P_2$ are two independent random polynomials as in $(\ref{rpo})$

Summing over $a \in \F_p$, we get
\begin{equation}\label{number}\sum_{a \in \F_p} \|\mu_a^{(n)}\|_2^2 = \E(N_p(P)),\end{equation}
where $P\in \Z[x]$ is a random polynomial with the same law as $P_1-P_2$ and $N_p(P)$ is the number of distinct roots of $P$ in $\F_p$. If $P=0$, then $N_p(P)=p$ of course, while if $P \neq 0$, we can use the crude upper bound $N_p(P) \leq n$. This yields
$$\sum_{a \in \F_p} \|\mu_a^{(n)}\|_2^2 \leq n + \P(P=0)p.$$
Together with a further trick exploiting the self-similarity of the law $$\mu_a^{(n+m)}= \mu_a^{(n)}*a^n.\mu_a^{(m)}$$ this will yield Theorem \ref{loglog}. 

In order to prove the cut-off in Theorem \ref{TVcutoff}, we need to perform a further averaging over primes and use the following key arithmetical ingredient, which is a special case of Chebotarev's theorem or just the prime ideal theorem: if $P\in \Z[x]$ is irreducible over $\Q$, then, on average over the primes $p$, $N_p(P)$ is equal to $1$. In other words as $X \to +\infty$
$$\E_{p}( N_p(P) )=  1 + \textnormal{error}$$ where $\E_p$ denotes the average over primes $p\leq X$ and the error tends to $0$ as $X \to +\infty$.
The quality of the error depends on what is known about the zeroes of the Dedekind zeta function $\zeta_K(z)$ of the number field $K:=\Q(x)/(P)$ near $z=1$. Assuming GRH, we get a very good error term. It is important to control the error term uniformly as $P$ varies, so explicit bounds in terms of the degree and discriminant of $P$ are necessary.

We proved in our previous paper \cite{breuillard-varju-irred} that $P$ is irreducible with high probability, and this implies that
the main contribution to the right hand side of \eqref{number} comes when $P$ is either irreducible or equal to $0$.
In the latter case we have $N_p(P)=p$, as we remarked above.
So we obtain
$$ \E_{p} \big(\sum_{a \in \F_p} \|\mu_a^{(n)}\|_2^2 ) \simeq  1+  \P(P=0) \E_{p}(p) $$
from $(\ref{number})$. 

Note that $\P(P=0)=\|\mu\|_2^{2n}$ since the $b_i$'s are i.i.d. and that $p\|\mu_a^{(n)}\|_2^2-1 = p\|\mu_a^{(n)}-u\|_2^2 \geq 0$. Now by the Markov inequality
$$ p\|\mu_a^{(n)} - u \|_2^2  \ll  p\|\mu\|_2^{2n}$$  
 holds for most $p$'s and most $a$'s.  In other words  (since $\|\mu_a^{(n)}\|_2 \geq \|\mu\|_2^{n}$) $$ \|\mu\|_2^{2n} p - 1 \leq  p\|\mu_a^{(n)}-u\|_2^2 \ll  \|\mu\|_2^{2n} p.$$ What this means is precisely that  we will see an $\ell^2$ cut-off exactly when $\|\mu\|_2^{2n} p$ becomes small, i.e. when $n \simeq \frac{1}{H_2(\mu)} \log p$. 

The cut-off in total variation of Theorem \ref{TVcutoff} is deduced from the $\ell^2$ cut-off estimate we have just described by restricting to a large part of the sample space where $P$ is roughly uniform. Finally Theorem \ref{uncond} is the result of a more precise analysis of what happens if we do not assume GRH and how we can make do with a weaker error term  and no irreducibility of $P$  using instead estimates on the number of zeroes of Dedekind zeta functions and Dobrowolski's Mahler measure lower bound to get a bound on the number of irreducible factors of $P$.

\subsection{Organization} The paper is organized as follows. In Section \ref{konyagin}, we discuss the general upper bound $(\ref{kon})$, which can be deduced from Konyagin's estimates. We include this for comparison with our new results. In Section \ref{multi}, we prove Theorem \ref{loglog} using a multiplicity trick exploiting the ``self-similarity'' feature of the chain $(\ref{chain})$. In Section \ref{ell2}, we prove Theorem \ref{ell2cutoff}, which establishes the cut-off in $\ell^2$ under GRH with a bounded cut-off window. Section \ref{total} is devoted to the proof of Theorem \ref{TVcutoff} and Corollary \ref{cor4}. Theorem \ref{uncond} will be established in Section \ref{un-bound}.  Finally, in Section \ref{lehmer-sec}, we mention mixing in $\ell^q$-norm, diameter bounds and connections with the Lehmer problem, and we indicate some directions of future research.

\section{The general upper bound}\label{konyagin} In this section, we discuss a general polylogarithmic upper bound on the mixing time. It was observed by Bukh, Harper, Helfgott and Lindenstrauss (see \cite{Hel-Arizona}*{Exercises 3.13 and 3.14}) that such a bound can be deduced easily from Konyagin's main lemma in \cite{Kon-RW}. More precisely we aim at the following.

\begin{theorem}[Konyagin, Bukh, Harper, Helfgott, Lindenstrauss]\label{konthm} There is an absolute constant $C>0$ such that the following holds. Let $p$ be a prime and $\mu$ a probability measure on $\F_p$ supported on  at least two elements. Consider the Markov chain $(\ref{chain})$ for a residue $a \in \F_p^{\times}$ with multiplicative order at least $C\log p (\log \log\log p)$. Then for every $\delta \in (0,\frac{1}{2})$,  we have
$$T(\delta) \leq T_2(\delta) \leq \frac{C}{1-\|\mu\|_2^2} \log(\delta^{-1}p)\log p (\log \log p)^{5}.$$
\end{theorem}

We include this result, because it provides a good comparison standpoint for the theorems of our paper that deal with mixing bounds for most values of the multiplier $a$ rather than all values as in Theorem \ref{konthm}. We will not use this result directly in this paper; however, a more general version of it (proved in Section 5.2 of \cite{breuillard-varju-irred}) is used crucially in our companion paper \cite{breuillard-varju-irred}, which is needed for Theorem \ref{TVcutoff}.

 We will use the following notation:  for $u\in \R$ we set $\widetilde{u}$ to be the unique representative of $u$ modulo $\Z$ in the interval $(-\frac{1}{2},\frac{1}{2}]$. It clearly depends only on the class of $u$ modulo $\Z$. For ease of notation, we identify the additive group of $\F_p$ with a subgroup of $\R/\Z$  by  viewing $x \mod p\Z$  as $x/p \in \R/\Z$.  We write $\log^{(n)}$ for the $n$-th iteration of $\log$. 

\begin{lemma}[Konyagin]\label{kon-lemma} There is an absolute constant $C\ge 1$ such that if $p$ is a prime and $a \in \F_p^{\times}$, for every $x_0\in \F_p^{\times}$ and every $m \ge C\log p (\log^{(2)} p)^4$, either $a$ has multiplicative order  at most $C\log p \log^{(3)} p$, or we have
\begin{equation}\label{eq:kon-conclusion}
\sum_{i=0}^m \Big(\widetilde{\frac{x_0a^i}{p}}\Big)^2 \ge \frac{1}{C \log^{(2)} p}.
\end{equation}
\end{lemma}

This lemma is a very slight improvement of \cite{Kon-RW}*{Lemma 6} and a special case of \cite{breuillard-varju-irred}*{Proposition 24}. The proof in \cite{breuillard-varju-irred} closely follows Konyagin's original argument in \cite{Kon-RW} and in the setting of Lemma \ref{kon-lemma} it does not introduce any new ideas.

For the convenience of the reader, we explain how to deduce Lemma \ref{kon-lemma} from \cite{breuillard-varju-irred}*{Proposition 24}. We take $M=m_1=D=1$, $Q=p$, $\b=x_0$ and $\a=a$ in \cite{breuillard-varju-irred}*{Proposition 24}, see \cite{breuillard-varju-irred}*{Section 5.1} for the notation. 
Taking $L=m$ in \cite{breuillard-varju-irred}*{Proposition 24}, we get that either \eqref{eq:kon-conclusion} holds, or there is a polynomial $P\in\Z[x]$ of degree at most $3\log p$ with Mahler measure at most
\begin{equation}\label{eq:Mahler-lower}
M(P)\le(\log p)^{30\log p/L}\le\exp(30C^{-1}(\log\log p)^{-3})
\end{equation}
such that $P(a)=0$. Recall that if $Q\in \Z[x]$ is a polynomial $Q=a_0+a_1x+\ldots+a_nx^n$, then $M(Q)$ is defined as \begin{equation}\label{eq:def-Mahler}M(Q)=|a_n| \prod_1^n \max\{1,|\alpha_i|\}\end{equation} where the $\alpha_i$'s are the complex roots of $Q$.  By the classical Dobrowolski bound \cite{Dob-Lehmer} on Mahler measure, either $$M(P)\ge \exp({c (\frac{\log\log \deg P}{\log \deg P})^3}) \ge \exp({c \frac{1}{(\log \log p)^3}})$$
for some absolute constant $c$, or $P$ is a product of cyclotomic polynomials. If we choose $C$ sufficiently large in \eqref{eq:Mahler-lower}, then we must have the second case, which makes the multiplicative order of $a$ less than $C\log p \log^{(3)} p$ (recall that the degree of the $n$-cyclotomic polynomial is Euler's totient function, which satisfies the estimate $\phi(n) \gg n/\log^{(2)} n$.)

\begin{proof}[Proof of Theorem \ref{konthm}] We aim at proving a good upper bound on the size of the Fourier coefficients of $\mu_a^{(n)}$. For $a \in \F_p$, set $\pi_a:\F_p[x] \to \F_p$ be the evaluation map $P \mapsto P(a)$. Recall that  $\mu_a^{(n)} = \pi_a(\mu^{(n)})$ is the law of $P(a)$, where $P(x)=b_0x^{n-1}+\ldots + b_{n-2}x +b_{n-1}$ is a random polynomial with law $\mu^{(n)}$, that is, the coefficients $b_i$'s are i.i.d. with law $\mu$ supported on $\F_p$.

The Fourier coefficient
$$\widehat{\mu_a^{(n)}}(\xi):= \E\Big(e(\frac{\xi P(a)}{p})\Big),$$
where $e(u):=\exp(2i\pi u)$, can be written, by independence of the $b_i$'s as
\begin{equation}\label{prodphi}\widehat{\mu_a^{(n)}}(\xi):= \prod_{i=0}^{n-1}\phi(\xi a^i),\end{equation}
where $$\phi(t)=\E(e(\frac{tb_0}{p}))= \sum_{x \in \F_p} \mu(x) e(\frac{tx}{p})$$ for $t\in \F_p$. We claim that
\begin{equation}\label{claim0}|\phi(t)|\leq \exp(- 4\sum_{x \neq 0} \nu(x) \Big(\widetilde{\frac{tx}{p}}\Big)^2),\end{equation}
where $\nu$ is the law of $b_0-b_1$.  Indeed, by easy calculus $\cos(2\pi u) \leq 1- 8\widetilde{u}^2$ for all $u \in \R$. Then
$$1-|\phi(t)|^2=1-\E(e(\frac{t(b_0-b_1)}{p}))= \sum_{x\neq0} \nu(x)(1-\cos(2\pi \frac{tx}{p})) \ge 8\sum_{x\neq 0}\nu(x)\Big(\widetilde{\frac{tx}{p}}\Big)^2$$
 and  (since $1-x\leq e^{-x}$ for all $x$) $(\ref{claim0})$ follows.
 
 Now $(\ref{prodphi})$ yields
$$|\widehat{\mu_a^{(n)}}(\xi)| \leq \exp(-4 \sum_{x \neq 0}\nu(x) \sum_{i=0}^{n-1} \Big(\widetilde{\frac{\xi x a^i}{p}}\Big)^2).$$
Note that $\nu(0)=\|\mu\|_2^2$.  Now splitting the sum of length $n$ into intervals of length $C \log p (\log^{(2)} p)^4$ to which we apply Lemma \ref{kon-lemma}, we obtain
$$|\widehat{\mu_a^{(n)}}(\xi)|^2 \leq \exp(-\frac{4n(1-\|\mu\|_2^2)}{C^2 \log p (\log^{(2)} p)^5} ),$$
provided $\xi\neq0$ and  $n \ge C\log p (\log^{(2)} p)^4$. This bound becomes $\leq \frac{\delta^2}{p}$, provided
$$n \ge \frac{C^2}{(1-\|\mu\|_2^2)} \log(\delta^{-1}p) \log p (\log^{(2)} p)^5.$$
If so, then by Parseval's identity
$$\|\mu_a^{(n)} - u\|_1^2 \leq p \|\mu_a^{(n)} - u\|_2^2 = \sum_{\xi \neq 0} |\widehat{\mu_a^{(n)}}(\xi)|^2 \leq \delta^2.$$
\end{proof}

\section{Multiplicity of Fourier coefficients}\label{multi}
In this section, we prove Theorem \ref{loglog}. The proof relies on the following observations about the Fourier coefficients of $\mu_a^{(n)}$ that come from the self-similar nature of the measure (a key feature of Bernoulli convolutions). For every $n,m \ge 0$, 
\begin{equation}\label{self-similar}\mu_a^{(n+m)}= \mu_a^{(n)}*a^n.\mu_a^{(m)}\end{equation}
where $*$ is the convolution product and $a^n.\mu_a^{(m)}$ is the image of $\mu_a^{(m)}$ under the multiplication map $x \mapsto a^nx$. This also plays a role in the paper \cite{CDG-RW}. From \eqref{self-similar}, it follows that for every $\xi \in \F_p$, if $0\leq r\leq s$,
$$ |\widehat{\mu_a^{(s)}}(\xi)| \leq  |\widehat{\mu_a^{(r)}}(\xi)|.$$ Moreover, for $\xi\neq0$, we have
\begin{equation} |\widehat{\mu_a^{(n+m)}}(\xi)| \leq |\widehat{\mu_a^{(n)}}(\xi)|\sup_{x \neq 0} |\widehat{\mu_a^{(m)}}(x)|,\label{firstline}\end{equation}
and iterating for any $k\ge 1$,
\begin{equation}\label{powers} |\widehat{\mu_a^{(km)}}(\xi)| \leq [\sup_{x \neq 0} |\widehat{\mu_a^{(m)}}(x)|]^k.\end{equation}
Similarly, using $\mu_a^{(n+m)}= \mu_a^{(i)}*a^i.\mu_a^{(n)}*a^{n+i}.\mu_a^{(m-i)}$  in place of  $(\ref{self-similar})$,
\begin{equation}
 |\widehat{\mu_a^{(n+m)}}(\xi)| \leq \inf_{i=0,\ldots,m} |\widehat{\mu_a^{(n)}}(a^i\xi)|. \label{secline}
\end{equation}
We can interpret $(\ref{secline})$ as a form of high multiplicity of the Fourier coefficients of $\mu_a^{(n)}$. Indeed it implies that if $\mu_a^{(n+m)}$ has one large Fourier coefficient, then $\mu_a^{(n)}$ must have at least $m$ large Fourier coefficients  provided $a$ has order at least $m$. From $(\ref{firstline})$ we have on the one hand
$$p\|\mu_a^{(2n+m)} -u \|_2^2=\sum_{\xi \neq 0} |\widehat{\mu_a^{(2n+m)}}(\xi)|^2 \leq \sum_{\xi \neq 0} |\widehat{\mu_a^{(n)}}(\xi)|^2  \sup_{\xi \neq 0} |\widehat{\mu_a^{(n+m)}}(\xi)|^2$$
and if $a$ has order $>m$ in $\F_p^{\times}$, by $(\ref{secline})$
\begin{equation}\label{self-improve}(m+1) \sup_{\xi \neq 0} |\widehat{\mu_a^{(n+m)}}(\xi)|^2 \leq \sup_{\xi \neq 0}\sum_{0 \le i \le m} |\widehat{\mu_a^{(n)}}(a^i\xi)|^2 \leq \sum_{\xi \neq 0} |\widehat{\mu_a^{(n)}}(\xi)|^2.\end{equation}
In particular, combining $(\ref{self-improve})$ with $(\ref{powers})$ we get for all $k,m,n\ge 0$:
\begin{equation}\label{goodbound}p\|\mu_a^{(n+k(m+n))} -u \|_2^2 \leq\sum_{\xi \neq 0} |\widehat{\mu_a^{(n)}}(\xi)|^2 (\frac{  \sum_{\xi \neq 0} |\widehat{\mu_a^{(n)}}(\xi)|^2}{m+1})^k \leq \frac{(p\|\mu_a^{(n)}\|_2^2)^{k+1}}{(m+1)^k}\end{equation} provided $a$ has order at least $m+1$ in $\F_p^\times$.

In conclusion we have shown that, when $a$ has sufficiently large multiplicative order,  the self-similarity $(\ref{self-similar})$ of the law implies an automatic decay of Fourier coefficients.  We can exploit this to prove Theorem \ref{loglog}.

\begin{proof}[Proof of Theorem \ref{loglog}]Recall our previous notation. For $a \in \F_p$,  $\pi_a:\F_p[x] \to \F_p$ is the evaluation map $P \mapsto P(a)$ and  $\mu_a^{(n)} = \pi_a(\mu^{(n)})$ is the law of $P_1(a)$, where $P_1(x)=b_0x^{n-1}+\ldots + b_{n-2}x +b_{n-1}$ is a random polynomial whose coefficients $b_i$'s are i.i.d. with law $\mu$ supported on $\F_p$. Then we make the crucial observation
$$\sum_{a \in \F_p} \|\mu_a^{(n)}\|_2^2 = \sum_{a \in \F_p} \P(P_1(a)=P_2(a)) = \E(N_p(P)),$$
where $N_p(P)$ is the number of distinct roots of $P$ in $\F_p$ and $P=P_1-P_2$ for two independent random polynomials $P_1,P_2$ distributed as above. If $P=0$, $N_p(P)=p$, while if $P\neq 0$, $N_p(P)\leq n$. Hence
$$\sum_{a \in \F_p} \|\mu_a^{(n)}\|_2^2 \leq p\|\mu\|_2^{2n} + n \leq 1+n,$$
if $n \ge \frac{1}{H_2(\mu)}\log p$. We thus choose to set $n=\lceil \frac{1}{H_2(\mu)}\log p \rceil$. By the Markov inequality, it follows that for all but an $\eps/2$-fraction of all residues $a \in \F_p$ we have
$$p\|\mu_a^{(n)}\|_2^2 \leq 2\eps^{-1}(1+ n).$$
Pick $m$ so that $10\eps^{-1}(n+1) \ge m+1 \ge 8\eps^{-1}(n+1)$. Then $p\|\mu_a^{(n)}\|_2^2/(m+1)\leq e^{-1}$ and choosing the least $k$ such that $e^{-k} \leq \delta^2 \eps / (n+1)$, we may apply $(\ref{goodbound})$ to conclude
$$\|\mu_a^{(M)} -u \|_1^2 \leq p\|\mu_a^{(M)} -u \|_2^2 \leq \delta^2$$
with 
$$M = n+k(m+n) = O(\eps^{-1} n \log(\eps^{-1}\delta^{-1}n)).$$ To apply $(\ref{goodbound})$, we need that the multiplicative order of $a$ is at least $m$, which holds for all but at most $m^2$ choices of $a$. Excluding these, the claim still holds for all but an $\e$-fraction of the residues $a\in\F_p$, provided $p$ is large enough depending on $\e$, which we may assume, otherwise the theorem is vacuous. The result follows.
\end{proof}

\section{Cut-off in $\ell^2$}\label{ell2}
We start by establishing an $\ell^2$ version of Theorem \ref{TVcutoff}. The method of proof will be adapted in the next section to establish Theorem \ref{TVcutoff}.  Recall that $H_2(\mu)= - \log  \|\mu\|_2^2$, where $\|\mu\|_2^2= \sum_{z \in \Z} \mu(z)^2$.

\begin{theorem}\label{ell2cutoff} Assume GRH. Fix a probability measure $\mu$ on $\Z$ supported on a finite set of at least two elements, and consider the Markov chain $(\ref{chain})$. Fix $\eps,\kappa>0$ and $\theta>0$. For $X>0$, let $n$ be the integer part of $\frac{1}{H_2(\mu)}(X+\theta)$. Set $D_{\mu,\kappa,\eps} =8(1+\kappa)\eps^{-2}\|\mu\|_2^{-2}$. Then for all but an $\eps$-fraction of all primes $p$ in $[e^X,(1+\kappa)e^X]$ and all but an $\eps$-fraction of all $a \in \F_p^\times$, we have
$$p\|\mu_a^{(n)} -u \|_2^2 \leq D_{\mu,\kappa,\eps}e^{-\theta},$$
provided $X \ge 2\log \eps^{-1}+C(\theta \log \theta)^2 + C$, where $C$ is a constant depending only on $\kappa$ and $\mu$. 
\end{theorem}

Note on the other hand that we have a universal lower bound$$p\|\mu_a^{(n)} -u \|_2^2 = p\|\mu_a^{(n)} \|_2^2 -1 \ge p\|\mu\|_2^{2n}-1 \ge e^{-\theta} -1$$ which is valid when $n=\lfloor \frac{1}{H_2(\mu)}(X+\theta) \rfloor$ for all $X>0$, $\theta \in \R$, and all $p \in [e^X,(1+\kappa)e^X]$.  In particular we see that the Markov chain is far from $\ell^2$-equidistribution if $\theta<0$, but becomes very close to equidistribution if $\theta>0$ is a fixed large number. Hence, for those primes $p$ and residues $a$ as in Theorem \ref{ell2cutoff} the $\ell^2$-mixing time satisfies $T_2(\delta)=\frac{1}{H_2(\mu)}(1+o_\delta(1))\log p$ as $p \to +\infty$, and hence that an $\ell^2$-cut-off takes place at $\frac{1}{H_2(\mu)}\log p$. The phase transition is very rapid and occurs within a small (bounded) window.

We will deduce Theorem \ref{ell2cutoff} from Proposition \ref{ell2cutoff-mix} below, which is the analogous result for the expectation of $p \|\mu_a^{(n)} -u \|_2^2$ (when averaging over $p$ and $a$) by a simple use of the Markov inequality. Given $n$ and $p$, we will say that a residue $a \in \F_p$ is \emph{admissible} if it is non-zero and not a root of a cyclotomic polynomial of degree at most $3n$.

\begin{proposition}\label{ell2cutoff-mix} Assume GRH. Fix a probability measure $\mu$ on $\Z$ supported on a finite set of at least two elements, and consider the Markov chain $(\ref{chain})$. Fix $\kappa>0$. For $n,X>0$,
$$ \E_{a,p}(p \|\mu_a^{(n)}\|_2^2 - p\|\mu\|_2^{2n})  \leq 1+  O(n^4X^2e^{-X/2}) + O(e^{-O(\sqrt{n}/\log n)}),$$
where $\E_{a,p}$ denotes the (weighted) average over all primes $p\in [e^X,(1+\kappa)e^X]$ and all admissible residues $a \in \F_p^\times$, and the implied constants depend only on $\kappa$ and $\mu$. 
\end{proposition}

By \emph{weighted average over primes}, we mean that we attribute the weight $\log p$ to each prime $p$ when counting.

\begin{proof}[Proof of Theorem \ref{ell2cutoff}] Let $Y=p \|\mu_a^{(n)}-u\|_2^2 = p\|\mu_a^{(n)}\|_2^2-1 $. By Proposition \ref{ell2cutoff-mix}, we have $$ \E_{a,p}(Y)\leq \E_{a,p}(p\|\mu\|_2^{2n}) +  O_\mu(\theta^4X^2 e^{-X/2}) + O(e^{-O_\mu(\sqrt{ X}/\log X)}).$$  Note that $\E_{a,p}(p\|\mu\|_2^{2n}) \leq  \alpha$, where $\alpha:= (1+\kappa)e^{-\theta}/\|\mu\|_2^2$.  Therefore $\E_{a,p}(Y) \leq 2\alpha$ provided $X\geq C_{\kappa,\mu}(1+(\theta \log \theta)^2)$.

We write  $\E_{a,p}(Y)=\E_{p}(\E_{a}(Y|p))$ and by the Markov inequality $$\P_p(\E_{a}(Y|p) \ge 4\alpha/\eps) \leq \eps/2.$$ Let $B$ be the set of primes $p\in [e^X,(1+\kappa)e^X]$ and $A$ the subset of those such that $\E_{a}(Y|p) > 4\alpha/\eps$. Then $$|A| X \leq \sum_{p\in A} \log p\leq\frac{\eps}{2}\sum_{p\in B} \log p \leq \frac{\eps}{2}(X+\log(1+\kappa))|B| \leq \eps X|B|,$$ provided $X\ge \log(1+\kappa)$. So $|A|/|B|\leq \eps$ . 

Now for $p \in B\setminus{A}$,  by the Markov inequality again, $\P_a(Y>8\alpha/\eps^2) \leq \eps/2$. And we deduce that for all but an $\eps$-fraction of primes $p \in [e^X,(1+\kappa)e^X]$, for all but an $\eps/2$-fraction of admissible residues $a \in \F_p$, we have $Y \leq 8\alpha/\eps^2$. On the other hand, there are at most $O(n^3)$ non-admissible residues. So again as soon as $X$ is large enough (say $X>2\log(\theta \eps^{-1}) +O(1)$, so that $n^3 \ll \eps e^X$), this represents at most an $\eps/2$-fraction of all residues modulo $p$.  This yields the upper bound in Theorem \ref{ell2cutoff}.
\end{proof}

We now turn to the proof of Proposition \ref{ell2cutoff-mix}. It will follow easily from the main theorem of \cite{breuillard-varju-irred}, whose proof assumed the validity of the Riemann hypothesis for Dedekind zeta functions of number fields, which we recall now.

\begin{theorem}[\cite{breuillard-varju-irred}*{Theorem 2}]\label{th:GRH} Assume GRH. Fix a probability measure $\nu$ on $\Z$ supported on a finite set of at least two elements. Let $\{A_i\}_{0\leq i\leq n}$ be independent random variables with common law $\nu$ and set $P=A_n x^n+\ldots+A_1x+A_0\in\Z[x]$. Then, with probability at least $1-\exp(-O_\nu(n^{1/2}/\log n))$, $$P=\Phi \wt P,$$ where  $\wt P$ is irreducible in $\Q[x]$,  $\deg \Phi =O_\nu(\sqrt{n})$ and $\Phi$ is a product of cyclotomic polynomials and a monomial $x^m$ for some $m\in\Z_{\ge 0}$.
\end{theorem}

Below we denote by $P=P_1-P_2$  the difference between two independent copies $P_1,P_2$ of a polynomial $a_{n-1} x^{n-1}+\ldots+a_1x+a_0$, where the $a_i$'s are i.i.d random variables chosen according to the law $\mu$. We also denote by $\wt P$ the quotient of $P$ by all its monomial and cyclotomic divisors. We will write $Ad$  for the set of admissible residues in $\F_p$. Also $N^{Ad}_p(P)$ will be the number of admissible roots of $P$.  We may write
\begin{align}\label{eq:l2sum}\sum_{a \in Ad} \|\mu_a^{(n)}\|_2^2 &= \sum_{a \in Ad} \P(P_1(a)=P_2(a))\\ \label{eq:l2sum20}&= |Ad| \P(P_1=P_2) + \E( 1_{\wt{P} \textnormal{ irred.}}N^{Ad}_p(P))+ O(n) \P(\wt{P} \textnormal{ reducible}).
\end{align}
The $O(n)$ on the right hand side comes from the trivial bound on the number of roots of $P$ in $\F_p$.

We denote by $\nu$ the law of $x-y$, where $x$ and $y$ are independent random variables with law $\mu$. This way, the coefficients of $P=P_1-P_2$ are i.i.d. and distributed according to $\nu$. Note that admissible roots of $P$ in $\F_p$ are roots of $\wt P$ in $\F_p$. Hence $$N^{Ad}_p(P) \leq N_p(\wt P),$$ where $N_p(\wt P)$ is the number of distinct roots of $\wt P$ in $\F_p$. Moreover, we observe that $\P(P_1=P_2)=\P(P=0)=\nu(0)^n=\|\mu\|_2^{2n}$. The $k$-th cyclotomic polynomial is of degree $\phi(k)$, and if $\phi(k) \leq 3n$, then $k =O( n \log \log n)$. In particular there are at most $O(n \log \log n)^2 =O(n^3)$ non-admissible roots in $\F_p$, so $p/|Ad| \leq 1+ O(n^3/p)$. In view of these observations, multiplying both sides of \eqref{eq:l2sum20} by $p/|Ad|$ yields
$$\E_{a}(p\|\mu_a^{(n)}\|_2^2) \leq p\|\mu\|_2^{2n} + \E(1_{\wt P \textnormal{ irred.}} N_p(\wt P)) + O(n^4/p) + O(n^4) \P(\wt{P} \textnormal{ reducible})$$
with absolute implied constants, where  $\E_{a}$ denotes the average over admissible residues.

Now using Theorem \ref{th:GRH}, we write
\begin{equation}\label{mas}\E_{a}(p\|\mu_a^{(n)}\|_2^2) \leq p\|\mu\|_2^{2n} + \E(1_{\wt P \textnormal{ irred.}} N_p(\wt P)) + O(n^4e^{-X})+O(\exp(-O(\frac{n^{1/2}}{\log n}))),\end{equation}
where the implied constants depend only on $\mu$.

We now perform the average over primes $p \in [e^X,(1+\kappa)e^X]$. 
The main point is the following well-known fact:  \emph{when $p$ varies among primes, the average number of distinct roots lying in $\F_p$ of a given non-constant polynomial $Q\in \Z[x]$ converges to the number of irreducible factors of $Q$.} This is an instance of the prime ideal theorem for number fields and effective estimates on the error term can be given in terms of the height and degree of the polynomial and the location of zeroes of the Dedekind zeta functions of the associated number fields. This observation was at the basis of the proof of Theorem \ref{th:GRH} which we gave in our previous article \cite{breuillard-varju-irred}, and it will be used here as well. 

If $Q\in \Z[x]$ is irreducible over $\Q$ and $p$ is a prime not dividing the discriminant of $Q$, then the number $N_p(Q)$ of roots of $Q$ in $\F_p$ coincides with the number $A_p(Q)$ of prime ideals of norm $p$ in the number field $\Q[x]/(Q)$ (see e.g. \cite{Coh-computational-NT}*{Theorem 4.8.13}). Moreover, under GRH, we have the following for all $X>0$
$$\sum_{p \in [e^X,(1+\kappa)e^X]} A_p(Q)\log (p)=\kappa e^X+O_\kappa(e^{X/2}X^2\log|\Delta_Q|),$$
where $\Delta_Q$ is the discriminant of $Q$. This is a classical result in analytic number theory originally due to Landau. For a proof in this exact form (when $\kappa=1$), see \cite{breuillard-varju-irred}*{Prop. 9}.

We may replace $A_p(Q)$ by $N_p(Q)$ in the above sum without seriously worsening the error term, because at most $O(\log |\Delta_Q|)$ primes can ever divide $\Delta_Q$, and both $A_p(Q)$ and $N_p(Q)$ are bounded by $d=\deg(Q)$. This implies that
\begin{equation}\label{pr:prime-sum-GRH}
\sum_{p \in [e^X,(1+\kappa)e^X]} N_p(Q)\log (p)=\kappa e^X+O_\kappa(e^{X/2}X^2d\log|\Delta_Q|).
\end{equation}

We can apply this estimate to $Q=x$ to count ordinary primes and also to $Q=\wt P$. Note that $\Delta_{\wt P}$ divides $\Delta_P$ and that $|\Delta_P| \leq (C_\mu n)^{2n}$, where $C_\mu$ is an upper bound on the absolute value of coefficients of $P$, i.e. $C_\mu=\max\{|z|, \mu(z)\neq 0\}$. Hence if $\wt P$ is irreducible, $(\ref{pr:prime-sum-GRH})$ yields
\begin{equation}\label{primeav}
 \E_p(N_p(\wt P))= 1+O_{\kappa,\mu}(e^{-X/2}X^2n^3),
\end{equation}
where we used the weighted average over primes in $[e^X,(1+\kappa)e^X]$.

Combining this with $(\ref{mas})$, we get
$$\E_{a,p}(p\|\mu_a^{(n)}\|_2^2) \leq \E_{p}(p\|\mu\|_2^{2n}) + 1 + O_{\kappa,\mu}(e^{-X/2}X^2n^3 +n^4e^{-X})+O_{\mu}(e^{-O_{\mu}(\frac{n^{1/2}}{\log n})})$$
and this yields the desired error term in Proposition \ref{ell2cutoff-mix}.

\section{Proof of cut-off in total variation}\label{total}
We now pass to the proof of Theorem \ref{TVcutoff}. The main idea is to use the Cauchy-Schwarz inequality to relate the total variation norm to the $\ell^2$ norm, and then to modify the proof of the $\ell^2$-cut-off given in the previous section by discarding rare events from the sample space. This is performed via the classical Chernoff-Hoeffding inequality, which says that if $X_1,\ldots,X_n$ are independent zero mean real random variables taking values in an interval $[a,b]$, then for all $t \ge 0$  $$\P(\frac{1}{n}|X_1+\ldots+X_n|\ge t)  \leq 2 \exp(-\frac{2nt^2}{(b-a)^2}).$$

We denote by $\mu^{(n)}$ the law of a random polynomial $P$ as in $(\ref{rpo})$ with independent integer random coordinates $b_i$'s with law $\mu$. Then setting $X_i=-\log \mu(b_i)- H(\mu)$ gives that for any $n,t>0$ the subset $A_n$ of all $P \in \Z[x]$ such that
\begin{equation}\label{mean}|\frac{1}{n}\log \mu^{(n)}(P) + H(\mu)|\leq t\end{equation}
has $\mu^{(n)}$-probability at least  $1-2e^{- 2t^2 n/C_\mu^2}$, where
$$C_\mu:=\max\Big(1,\log \sup_{x \in \Supp(\mu)} \mu(x) - \log \inf_{x \in \Supp(\mu)} \mu(x)\Big).$$

Let $\mu_{|A_n}^{(n)}$ be the measure (possibly of mass $<1$) obtained by restricting $\mu^{(n)}$ to the event $A_n$. Let $\pi_a:\Z[x] \to \F_p$ be the evaluation map sending $P$ to $P(a)$, so that $\mu_a^{(n)}=\pi_a(\mu^{(n)})$. We may write
\begin{align}\|\mu_a^{(n)} - u\|_1 \leq& \|\pi_a(\mu_{|A_n}^{(n)}) - u\|_1 +\|\pi_a(\mu_{|A_n}^{(n)} - \mu^{(n)})\|_1\nonumber\\ \leq& \|\pi_a(\mu_{|A_n}^{(n)}) - u\|_1 +\|\mu_{|A_n}^{(n)} - \mu^{(n)}\|_1\nonumber\\ =& \|\pi_a(\mu_{|A_n}^{(n)}) - u\|_1 + \mu^{(n)}(A_n^c).  \label{1212}\end{align}
On the other hand, by Cauchy-Schwarz,
\begin{equation}\label{CS} \|\pi_a(\mu_{|A_n}^{(n)}) - u\|_1^2 \leq p  \|\pi_a(\mu_{|A_n}^{(n)}) - u\|_2^2. \end{equation} 
We write
\begin{align} p\|\pi_a(\mu_{|A_n}^{(n)}) - u\|_2^2 &= p\|\pi_a(\mu_{|A_n}^{(n)})\|_2^2 -2\mu^{(n)}(A_n) +1\nonumber\\ &=  p\|\pi_a(\mu_{|A_n}^{(n)})\|_2^2 -1 +2\mu^{(n)}(A_n^c).\label{eqqqq}
\end{align}
Combining \eqref{1212}, \eqref{CS} and \eqref{eqqqq}, we get
\begin{equation}\label{eq:final}
\|\mu_a^{(n)}-u\|_1^2\le p\|\pi_a(\mu_{|A_n}^{(n)})\|_2^2 -1 +5\mu^{(n)}(A_n^c).
\end{equation}

Similarly to \eqref{eq:l2sum}, we write
\begin{equation}\sum_{a \in Ad} \|\pi_a(\mu_{|A_n}^{(n)})\|_2^2 =  \E_{|A_n}^{(n)}(N_p^{Ad}(P)),\label{eq:l2sum2}\end{equation}
where, as before, we write $P=P_1-P_2$ for the difference of two independent copies $P_1$, $P_2$ of a random polynomial with law $\mu^{(n)}$, and $N_p^{Ad}(P)$ is the number of admissible roots of $P$ in $\F_p$. Also, $\E_{|A_n}^{(n)}(X)$ is a shorthand for $\E^{(n)}(X\cdot 1_{A_n\times A_n})$ and $\E^{(n)}$ is the expectation for the probability law $\mu^{(n)}\times\mu^{(n)}$. Writing $\wt P$ again for the quotient of $P$ by all its monomial  and cyclotomic factors, we have
$$ \E_{|A_n}^{(n)}(N_p^{Ad}(P)) \leq |Ad| \P^{(n)}_{|A_n}(P=0)  + \E^{(n)}( 1_{\wt{P} \textnormal{ irred.}}N^{Ad}_p(P))+ O(n) \P^{(n)}(\wt{P} \textnormal{ reducible}).$$
We plug this into \eqref{eq:l2sum2}, multiply both sides by $p/|Ad| \leq 1+ O(n^3e^{-X})$, and apply Theorem \ref{th:GRH} as before,
$$\E_{a}(p \|\pi_a(\mu_{|A_n}^{(n)}) \|_2^2) \leq p\|\mu_{|A_n}^{(n)}\|^2_2 + \E^{(n)}(1_{\wt P \textnormal{ irred.}} N_p(\wt P)) + O(n^4e^{-X})+ O(\exp(-O(\frac{n^{1/2}}{\log n}))),$$
where $\E_a$ denotes the average over admissible residues, and the implied constants depend only on $\mu$.
We average over primes as before, and using $(\ref{primeav})$, we get
\begin{equation}\label{mas2}\E_{a,p}(p \|\pi_a(\mu_{|A_n}^{(n)}) \|_2^2) \leq \E_p(p\|\mu_{|A_n}^{(n)}\|^2_2) + 1 +O_{\kappa,\mu}(e^{-X/2}X^2n^4)+O_{\mu}(e^{-O_{\mu}(\frac{n^{1/2}}{\log n})}). \end{equation}

Now recall that $n$ is the integer part of $\frac{1}{H(\mu)}(X+\theta \sqrt{X})$. We combine \eqref{eq:final} and \eqref{mas2} to write
\begin{equation}\label{last}\E_{a,p}(\|\mu_a^{(n)} - u\|_1^2)\leq  \E_p(p\|\mu_{|A_n}^{(n)}\|^2_2)  + 5 \mu^{(n)}(A_n^c) + O_{\kappa,\mu}(\theta^4 X^2 e^{-X/2})+O_\mu(e^{-O_\mu(\frac{\sqrt{X}}{\log X})}).\end{equation}
Using $(\ref{mean})$,
\begin{align*}p\|\mu_{|A_n}^{(n)}\|^2_2&=\sum_{\omega \in A_n} p\mu^{(n)}(\omega)^2 \leq p e^{-(H(\mu)-t)n}\\ &\leq (1+\kappa)e^X e^{-(H(\mu)-t)(\frac{1}{H(\mu)}(X+\theta\sqrt{X})-1)}\end{align*}
We set $t=\frac{\theta H(\mu)}{2\sqrt{X}}$ and obtain
$$p\|\mu_{|A_n}^{(n)}\|^2_2 =O_{\kappa,\mu}(e^{-\frac{\theta}{2}\sqrt{X}} e^{\frac{\theta^2}{2}}),$$
while
$$ \mu^{(n)}(A_n^c) \leq 2e^{-2t^2 n/C_\mu^2} \leq 2 e^{-\frac{H(\mu)}{4C_\mu^2} \theta^2 }.$$
Combining the last two estimates with $(\ref{last})$, we obtain
\begin{equation}\label{lastt}\E_{a,p}(\|\mu_a^{(n)} - u\|_1^2)\leq O_{\kappa,\mu}(e^{\frac{\theta^2}{2}}e^{-\frac{\theta}{2}\sqrt{X}})+ 10e^{-\frac{H(\mu)}{4C_\mu^2} \theta^2 } +O_{\kappa,\mu}(\theta^4X^2e^{-X/2})+O_\mu(e^{-O_\mu(\sqrt{X}/\log X)}).\end{equation}

Writing $\a=10e^{-\frac{H(\mu)}{4C_\mu^2}\theta^2 }$, we get
\[
\E_{a,p}(\|\mu_a^{(n)} - u\|_1^2)\leq2\a,
\]
provided $X>C(\theta^2\log\theta)^2+C$ for a suitably large $C$ (dependent on $\mu,\kappa$ only). Now applying the Markov inequality twice in exactly the same way as in the proof of Theorem \ref{ell2cutoff}, we find that for all but an $\e$ proportion of the primes $p$ in $[e^X,(1+\kappa)e^X]$,
and for all but an $\e$ proportion of $a\in\F_p$, we have
\[
\|\mu_a^{(n)} - u\|_1^2\leq 8\a/\e^2.
\]
Taking square roots and dividing by 2, we get the desired estimate. This ends the proof of Theorem \ref{TVcutoff}.

\begin{proof}[Proof of Corollary \ref{cor4}]
We begin with the proof of $(\ref{ent-b})$. Set $t=H(\mu)\theta/\sqrt{\log p}$ and consider $(\ref{mean})$. We have $$|\pi_a(A_n)|\min_{\omega \in A_n} \mu^{(n)}(\omega) \leq \mu^{(n)}(A_n)\leq 1.$$ Hence $$|\pi_a(A_n)|\leq e^{n(H(\mu)+t)}.$$

Since $n = \lfloor(\log p-\theta\sqrt{\log p})/H(\mu)\rfloor$, this means that most of the walk is confined to a small part of $\F_p$, and hence
\begin{align*}\|\mu_a^{(n)}-u\|_{TV} &\ge |\mu_a^{(n)}(\pi_a(A_n)^c)-u(\pi_a(A_n)^c)|\\ &\ge \frac{|\pi_a(A_n)^c|}{p} - \mu_a^{(n)}(\pi_a(A_n)^c)\\ &\ge 1- \frac{1}{p}e^{n(H(\mu)+t)}-\mu^{(n)}(A_n^c)\\ &\ge1- \frac{1}{p}e^{(\log p-\theta\sqrt{\log p})H(\mu)^{-1}(H(\mu)+H(\mu)\theta/\sqrt{\log p})}-2e^{-2t^2n/C_\mu^2}\\ &\ge1- \frac{1}{p}e^{\log p-\theta^2}-2e^{-2H(\mu)^2\theta^2(\log p)^{-1}\log p/2H(\mu)C_\mu^2}\\ &\ge 1- e^{-\theta^2} - 2e^{-\frac{H(\mu)}{C_\mu^2}\theta^2}.
\end{align*}
This establishes $(\ref{ent-b})$.

We now move on to the proof of Corollary \ref{cor4}.  Set $\theta=D_\mu \sqrt{\log \eps^{-1}}$ for some constant $D_\mu>0$ to be determined later. First, the lower bound for $T(\delta)$ follows easily from $(\ref{ent-b})$. Indeed, if $n \leq (\log p  - \theta \sqrt{\log p})/H(\mu)$, then  $(\ref{ent-b})$ implies that $$\|\mu_a^{(n)}-u\|_{TV} \ge 1- e^{-O_\mu(\theta^2)} \ge 1-\eps \ge \delta$$
provided $D_\mu$ is large enough.

Now for the upper bound, set $\kappa=1$ and, for every integer $m\ge 1$, $X=\log 2^m$ in Theorem \ref{TVcutoff}.  Let  $E_m$ be the set of exceptional primes given by Theorem \ref{TVcutoff} in the interval $[2^m,2^{m+1}]$. If $p\notin E_m$ but $\log 2^m \ge  \theta^5$, then Theorem \ref{TVcutoff} yields:
 $$\|\mu_a^{(n)}-u\|_{TV} \leq 8\eps^{-1}e^{-\frac{H(\mu)}{16C_\mu^2}\theta^2} \leq \eps \leq \delta$$
provided $D_\mu$ is large enough and $n \ge (\log p  + \theta \sqrt{\log p})/H(\mu)$. This establishes the corollary with  $E:=\cup_{m \ge \theta^5} E_m$. Note that $E$ has density at most $\eps$ among all primes, since $E_m$ has density $\leq \eps$ among primes in $[2^m,2^{m+1}]$ for all $m$.
\end{proof}

We also state another variant of Corollary \ref{cor4}. Its proof is entirely similar, but in the application of Theorem \ref{TVcutoff}, one needs to take $\theta^2= \sqrt{X}/(D_\mu \log X)$ for a suitable constant $D_\mu$ and $\eps=\exp(H(\mu)\theta^2/32C_\mu^2)$. We leave the details to the reader.

\begin{corollary}\label{cor5} Assume GRH. There is an exceptional set of primes $F$ of density $0$, such  that for all primes $p\notin F$, for every $\eta \in (0,1)$  and for all but an $e^{-O_\mu(\sqrt{\log p}/\log \log p)}$-fraction of all $a \in \F_p^\times$ we have:
$$|T(\delta) - \frac{1}{H(\mu)}\log p| \leq \frac{\eta}{H(\mu)} \log p,$$
for all $\delta \in (\eta,1-\eta)$, provided $\log p > C(\mu,\eta)$.

Moreover, the density of the exceptional primes $F$ can be estimated as
\[|\{p \in F \cap [1,e^X]\}| = O_\mu(e^{-O_\mu(\sqrt{X}/\log X)}) \frac{e^X}{X}.\]
\end{corollary}

\section{Unconditional upper bound}\label{un-bound}
In this section, we prove Theorem \ref{uncond}. Its proof, which no longer assumes GRH, elaborates, on the one hand, on the bounds obtained in Section \ref{ell2}, and, on the other hand, on the multiplicity trick already used in Section \ref{multi}. First we show the following.

\begin{proposition}\label{uncond-av} Fix a probability measure $\mu$ on $\Z$ supported on a finite set of at least two elements, set $H_2(\mu)=-\log \|\mu\|_2^2$ and consider the Markov chain $(\ref{chain})$. For $X$ be a sufficiently large (depending on $\mu$) number and let $n$ be the integer part of $\frac{1}{H_2(\mu)}X$. Then
$$\sum_{p\in [e^{X/2},e^X]} \E_{a}(\|\mu_a^{(n)} \|_2^2)  = O_\mu\big((\log X)^5\big),$$ where the  average $\E_a$ is taken over all admissible residues $a \in \F_p^\times$. 
\end{proposition}

Before we start the proof of this proposition, let us go back to the setting of Section \ref{ell2}, where we averaged over primes in order to get an upper bound on the $\ell^2$ norm. Without GRH at our disposal, we can no longer guarantee that a random polynomial is a product of a single irreducible polynomial together with monomial and cyclotomic factors. However, we can easily give an upper bound on the number of irreducible factors. Indeed, if $P \in \Z[x]$ and $P=\Phi\wt P$, $\wt P= P_1\ldots P_k$, where the $P_i$'s are the non-cyclotomic irreducible factors, then
$$M(P) = \prod_1^k M(P_i),$$
where $M(P)$ is the Mahler measure of $P$, see \eqref{eq:def-Mahler} for the definition.
By the classical Dobrowolski bound \cite{Dob-Lehmer} on Mahler measure, $$M(P_i)\ge e^{c (\frac{\log\log d_i}{\log d_i})^3} \ge e^{c \frac{1}{(\log n)^3}},$$
where $d_i=\deg P_i$ and $c$ is an absolute constant. In particular,
$$M(P) \ge \exp( ck/(\log n)^3).$$
On the other hand, $M(Q)\leq (\sum |a_i|^2)^{1/2}$ if $Q=\sum_i a_ix^i \in \Z[x]$. In our setting, $P=P_1-P_2$, where $P_1,P_2$ are distributed according to $\mu^{(n)}$, so $|a_i| \leq 2H$ with $H:=\max\{|z|, \mu(z)\neq 0\}$. Hence $M(P) \leq 2H\sqrt{n}$, which, compared with the previous bound gives
\begin{equation}\label{irred-est}k = O_\mu( (\log n)^4).\end{equation}

We may thus perform the average over primes, as in Section \ref{ell2}, in order to give a reasonably good upper bound on $p\|\mu_a^{(n)}\|^2_2$. To achieve this without GRH requires a slightly different kind of prime averaging as reflected in the statement of Theorem \ref{uncond}.

\begin{proof}[Proof of Proposition \ref{uncond-av}] Proceeding as in $(\ref{mas})$,  we write
\[\E_{a}(p\|\mu_a^{(n)}\|_2^2) =  p\|\mu\|_2^{2n} + \frac{p}{|Ad|}\E(1_{ P \neq 0} N^{Ad}_p( P)),\]
where $Ad \subset \F_p$ is the set of admissible residues,
and $P=P_1-P_2$, where $P_1,P_2$ are distributed according to $\mu^{(n)}$.
Since $p-|Ad|=O(n^3)$, we get
\begin{equation}\label{maaas}\E_{a}(p\|\mu_a^{(n)}\|_2^2) =  p\|\mu\|_2^{2n} + \E(1_{P \neq 0} N^{Ad}_p(P))+O(n^4/p).\end{equation}
We are now in a position to sum over primes. Assuming GRH we had  $(\ref{pr:prime-sum-GRH})$ with an excellent error term. Here we no longer assume GRH, but we can nevertheless make do with a weaker estimate. We have:

\begin{lemma} \label{av-bound} If $Q \in \Z[x]$ is irreducible, of degree $d$ and discriminant $\Delta_Q$, and if, as earlier, $N_p(Q)$ denotes the number of distinct roots of $Q$ in $\F_p$, then for all $X>0$,
$$\sum_{p \in [e^{X/2},e^X]} \frac{N_p(Q)}{p}=O\big( \frac{\log(2|\Delta_Q|)}{d}(1+\frac{d^4}{X^4})\big),$$ where the implied constant is absolute.
\end{lemma}

This lemma is unconditional and should be compared to $(\ref{pr:prime-sum-GRH})$, which assumed GRH. Recall  that for all large enough $X$, $\sum_{p \in [e^{X/2},e^X]} \frac{1}{p} \in [c,1/c]$ for some absolute constant $c \in (0,1)$, as follows directly from the prime number theorem or from Chebychev's bounds.  Thus the lemma can be interpreted as a way to express the fact that $N_p(Q)$ remains bounded by $O(\log(2|\Delta_Q|)/d)$ on average over the primes in large windows $[\sqrt{x},x]$, $\log x>d$, where the average assigns a weight $1/p$ to each prime $p$. 

We postpone the proof of this lemma, and, first, explain how to complete the proof of Proposition \ref{uncond-av}.
We write $P=\Phi\wt P$, $\wt P= P_1\ldots P_k$, where the $P_i$'s are the non-cyclotomic irreducible factors. 
From $(\ref{maaas})$, summing over primes in $[e^{X/2},e^X]$, we get:
\begin{equation}\label{sumsum}\sum_{p \in [e^{X/2},e^X]}\E_{a}(\|\mu_a^{(n)}\|_2^2) \leq   e^X \|\mu\|_2^{2n} + \E(1_{ P \neq 0} \sum_{p \in [e^{X/2},e^X]}\frac{N_p(\wt P)}{p})+O(e^{-X/2}n^4).\end{equation}
Now, by Lemma \ref{av-bound}, bearing in mind that $n=\lfloor \frac{1}{H_2(\mu)}X \rfloor$,
$$\sum_{p \in [e^{X/2},e^X]} \frac{N_p(\wt P)}{p} \le \sum_{i=1}^k \sum_{p \in [e^{X/2},e^X]} \frac{N_p(P_i)}{p} = \sum_{i=1}^k  O_\mu( \frac{\log |\Delta_{P_i}|}{\deg P_i}).$$
On the other hand, by Mahler's bound \cite{Mah-discriminant}*{Theorem 1}, $$(\log |\Delta_{P_i}|)/\deg(P_i) \leq \log \deg(P_i) + 2\log M(P_i).$$ But $M(P_1)\ldots M(P_k)=M(P) \leq 2H\sqrt{n}$, where $H:=\max\{|z|, \mu(z)\neq 0\}$.  Hence we get
\begin{align*}\sum_{p \in [e^{X/2},e^X]} \frac{N_p(\wt P)}{p} &=  \sum_{i=1}^k  O_\mu( \frac{\log |\Delta_{P_i}|}{\deg P_i})\\ &\leq O_\mu( 1+ (k+1) \log n) = O_\mu((\log n)^5),\end{align*} where we used $(\ref{irred-est})$ in the last line. Plugging this back into $(\ref{sumsum})$, we obtain the desired estimate.

\end{proof}

\begin{proof}[Proof of Lemma \ref{av-bound}]
The proof is a straightforward variation of the proof of Proposition 13 in \cite{breuillard-varju-irred} and for this reason we shall be brief.

The argument relies on the explicit formula in analytic number theory relating the zeroes of the Dedekind zeta function and the count of prime ideals in a number field with explicit error terms. This is a standard tool in analytic number theory, which is extensively discussed in the literature in various settings. We will refer to \cite{breuillard-varju-irred}, because it contains the formulae in the
precise from we need them. 

Let $K$ be a number field, $d$ its degree, and $\Delta_K$ its discriminant. Let $A(n)$ be the number of prime ideals with norm $n$. Given a $C^2$ compactly supported function $g$ on $(0,+\infty)$, and $s \in \C$ we set $$G(s)=\int_\R \exp(su)g(u)du.$$
The explicit formula (see e.g. \cite{breuillard-varju-irred}*{Theorem 10}) asserts that 
\begin{equation}\label{explicit}\sum_{n,m \ge 1} A(n) \log(n) g(\log(n^m)) = G(1) - \sum_{\rho} G(\rho),\end{equation}
where the last sum on the right hand side extends over all zeroes $\rho$ of the Dedekind zeta function of $K$ counted with multiplicity. 

First, we focus on estimating the right hand side of $(\ref{explicit})$. We make the following choice for $g$. We set $g(u)=e^{-u}\frac{8}{X}\chi^{*4}(\frac{8u}{X}-6)$, where $\chi=1_{[-\frac{1}{2},\frac{1}{2}]}$ is the indicator function of the unit length interval $[-\frac{1}{2},\frac{1}{2}]$ and $\chi^{*4}$ is its four-fold convolution product. This corresponds to the choice $g=g_{X,k}$ with $k=4$ in \cite{breuillard-varju-irred}*{Lemma 11}. In that lemma, the following elementary properties of $G(s)$ were verified for all complex $s$ with $\Re(s)\leq 1$,
\begin{align} |G(s)| \leq e^{(\Re(s)-1)X/2},\label{qq}\\ |G(s)| \leq \big( \frac{16}{X|s-1|} \big)^4.\label{qqq}
\end{align}

We also have at our disposal standard explicit bounds on the number (counting multiplicity) of zeroes of the zeta function $\zeta_K$ near $1$ and elsewhere. All zeros satisfy $\Re\rho<1$, and for every  $0\le r\le 1$, we have
\begin{equation}\label{estless1}
|\{\rho:\zeta_K(\rho)=0, |1-\rho|<r\}|\le\frac{3}{2}+3r\log|\Delta_K|,
\end{equation}
while there is an absolute constant $C>0$, such that for every $r>1$, 
\begin{equation}\label{estmore1}
|\{\rho:\zeta_K(\rho)=0, |1-\rho|<r\}|\le C \log|\Delta_K|+C d r\log r.
\end{equation}
For those we refer to \cite{breuillard-varju-irred}*{Lemma 12} or \cite{Sta-Dedekind}. 

Note also that $\log |2\Delta_K| \ge cd$ for some absolute constant $c>0$ by Minkowski's bound. Then combining $(\ref{estless1})$ with $(\ref{qq})$, we get
\begin{equation}\label{firstG} \sum_{\rho, |1-\rho|<d^{-1}} |G(\rho)| = O(\frac{\log 2|\Delta_K|}{d}).\end{equation}

We now partition the remainder of the complex plane in annuli around $1$ as follows for $j\in \Z_{\ge 0}$
\[R_j:=\{\rho:\zeta_K(\rho)=0, 2^{j}d^{-1}\le |1-\rho|< 2^{j+1}d^{-1}\}.\]
Using $(\ref{estless1})$ and $(\ref{estmore1})$,  this yields  that for all $j \in \Z_{\ge 0}$,
\[
|R_j|=O(4^jd^{-1}\log 2|\Delta_K|).
\]
Now, by $(\ref{qqq})$,
$$\sum_{\rho \in R_j, j\ge 0} |G(\rho)| \leq \sum_{j\ge 0}  |R_j| (\frac{16}{Xd^{-1} 2^j})^4 \leq   \sum_{j\ge 0}  O(d^{-1}\log 2|\Delta_K|) 4^{-j}(\frac{16}{Xd^{-1}})^4 .$$

Now, combining this estimate with $(\ref{firstG})$, yields
\begin{equation}\label{Gbound}
|G(1)-\sum_{\rho} G(\rho)|\leq1+\sum_{\rho} |G(\rho)| = O\big(\frac{\log 2|\Delta_K|}{d}(1+\frac{d^4}{X^4})\big).\end{equation}

Now, we turn to the left hand side of \eqref{explicit}.
Observe that $g$  is supported in the interval $[X/2,X]$, and that $0\leq g(u) \leq e^{-u}$ for all $u>0$. Therefore, the main contribution in the above sum on the left hand side comes from prime ideals of prime norm. Indeed, it is shown in the proof of \cite{breuillard-varju-irred}*{Proposition 13} (see the bound $(2.5)$ therein and subsequent estimates) that
\begin{equation}\label{pp}\big|\sum_{n,m \ge 1} A(n) \log(n) g(\log(n^m)) -\sum_{p \textnormal{ prime}} A(p) \log(p) g(\log(p)) \big| = O(dXe^{-X/4}),\end{equation}
where the implied constant is absolute.

We apply \eqref{explicit} to the number field $K=\Q[x]/(Q)$. We note that $\Delta_K$ divides $\Delta_Q$, and, as already mentioned in Section \ref{ell2}, $A(p)$ may differ from $N_p(Q)$ only for those primes $p$ that divide the discriminant $\Delta_Q$. This accounts for at most $O(\log |\Delta_Q|)$ primes, hence:
\begin{equation}\label{compAN}
\big|\sum_{p \textnormal{ prime}} N_p(Q) \log(p) g(\log(p)) -\sum_{p \textnormal{ prime}} A(p) \log(p) g(\log(p)) \big| = O(de^{-{X/2}}\log |\Delta_Q|)
\end{equation}

Putting $(\ref{explicit})$, $(\ref{Gbound})$, $(\ref{pp})$ and $(\ref{compAN})$ together, we obtain
$$\sum_{p \textnormal{ prime}} N_p(Q) \log(p) g(\log(p)) = O\big(\frac{\log 2|\Delta_Q|}{d}(1+\frac{d^4}{X^4})\big).$$

Now, observe from our definition of $g$ that if $p \in [e^{3X/5},e^{4X/5}]$, then $g(\log p) \ge c \frac{1}{pX}$ for some absolute constant $c>0$. Hence setting $Y=4X/5$, we get
$$\sum_{p \in [e^{3Y/4},e^{Y}]}\frac{N_p(Q)}{p} = O\big(\frac{\log 2|\Delta_Q|}{d}(1+\frac{d^4}{Y^4})\big).$$
Applying this bound to $Y=X$, $Y=3X/4$ and $Y=9X/16$, we conclude that, indeed,
$$\sum_{p \in [e^{X/2},e^{X}]} \frac{N_p(Q)}{p} = O\big(\frac{\log 2|\Delta_Q|}{d}(1+\frac{d^4}{X^4})\big),$$
as desired.

\end{proof}

\begin{proof}[Proof of Theorem \ref{uncond}] Let $C_\mu$ be the implied constant in the big-$O$ appearing in Proposition \ref{uncond-av} and let $B$ be the set of primes $p$ such that 
$$\E_a(p \|\mu_a^{(\lfloor 2\frac{1}{H_2(\mu)} \log p \rfloor)}\|_2^2) > 4C_\mu \sqrt{\log p}(\log \log p)^{-1}.$$
By Proposition \ref{uncond-av}, using that $n \mapsto \|\mu_a^{(n)}\|_2$ is non-increasing, we get
$$\sum_{p\in [e^{X/2},e^X]} \frac{1_{p \in B}}{p} \leq \frac{(\log X)^6}{\sqrt{X}}.$$

Let $p \notin B$ and set $n=\lfloor 2\frac{1}{H_2(\mu)} \log p \rfloor$. Then $$\E_a(p \|\mu_a^{(n)}\|_2^2) \leq 4C_\mu \sqrt{\log p}(\log \log p)^{-1}.$$ In particular, by the Markov inequality, for all but an $\eps/2$-fraction of all admissible residues, and hence for all by an $\eps$-fraction of all residues $a \in \F_p$, we have
 $$p \|\mu_a^{(n)}\|_2^2\leq 8\eps^{-1}C_\mu \sqrt{\log p}(\log \log p)^{-1}.$$ 
Using $(\ref{goodbound})$, with $k=1$, and $m=\lfloor \frac{1}{H_2(\mu)} \log p \rfloor$, we may now bound the right hand side, and obtain
$$\|\mu_a^{(\lfloor 5\frac{1}{H_2(\mu)} \log p \rfloor)} - u\|_{1}^2\leq p\|\mu_a^{(2n+m)} - u\|_2^2 =O_\mu( \frac{\eps^{-2}}{(\log \log p)^2}),$$
and the claimed bounds on mixing times follow.
\end{proof}

\section{Concluding remarks}\label{lehmer-sec}

\subsection{Mixing and cut-off in $\ell^q$}\label{ellq}
It is natural to ask about mixing times $T_q(\delta)$ in $\ell^q$-norm for any $q \in [1,+\infty]$. Here
$$T_q(\delta):=\inf\{n \in \N, \|\mu_a^{(n)}-u\|_q \leq \delta \|u\|_q\}.$$ As is well-known, by convexity, the function  $q \mapsto \|\mu_a^{(n)}-u\|_q/\|u\|_q$ is non-decreasing. In particular, $q \mapsto T_q(\delta)$ is non-decreasing (see e.g. \cite{peres-levin}*{Section 4.7}). On the other hand,  Cauchy-Schwarz implies that $\|\mu_a^{(2n)}-u\|_\infty \leq \|\mu_a^{(n)}-u\|_2^2$ for all $n$. And it follows that for all $q   \in [1,+\infty]$ we have \begin{equation}\label{tq}T_q(\delta) \leq 2T_2(\delta^2).\end{equation}
This means that we can apply our main theorems to get a good upper bound on $T_q(\delta)$. 

Regarding lower bounds, we have, as before, the trivial lower bound $\|\mu_a^{(n)}-u\|_q/\|u\|_q \ge \|\mu^{(n)}\|_q/\|u\|_q -1 =  \|\mu\|_q^n/\|u\|_q -1$, which implies that $$T_q(\delta) \ge \frac{1}{H_q(\mu)} (\log p - O_q(\delta)),$$ where $H_q(\mu)=\frac{q}{q-1} \log \|\mu\|^{-1}_q$ is the R\'enyi entropy of order $q$ of $\mu$.

If $\mu$ is uniform (i.e. $\mu(x)=\|\mu\|_\infty$ for all $x \in \Supp(\mu)$), then $H_q(\mu)=H_\infty(\mu)=H(\mu)$ for all $q$. In particular, we see that mixing in $\ell^1$ and in $\ell^2$ occur at the same time, and hence so does mixing in $\ell^q$ norm for any $q \in [1,2]$. Consequently, in this situation, cut-off happens in $\ell^q$-norm in the setting of Corollary \ref{cor4} for instance.  However, when $q >2$, it seems less clear how to establish an analogous cut-off in $\ell^q$-norm.

\subsection{The exceptional set of multipliers} As mentioned in the introduction, the multiplier $a$ must be of large enough multiplicative order in $\F_p^{\times}$  for the chain $(\ref{chain})$ to equidistribute in $O(\log p)$ steps, let alone satisfy a cut-off at $(\log p)/H(\mu)$ as in Corollary \ref{cor4}. Even a multiplicative order of size more than $\log p$ is not enough; indeed, Chung-Diaconis-Graham \cite{CDG-RW} show that when $a=2$ and $\mu$ is uniform on $\{-1,0,1\}$, if $p$ is of the form $2^n-1$, then at least $c n \log n$ steps are necessary for mixing to occur. In this case, $a$ has order $n \simeq \log p$. Their example can easily be modified to obtain exceptional $a$'s with order at least $(\log p)^2$. 

Now, as far as cut-off at $(\log p)/H(\mu)$  is concerned, the exceptional set of multipliers is even larger, and contains residues of large multiplicative order. Indeed, as remarked in \cite{hildebrand-lower-bound}, when $a=2$, the support of $\mu_a^{(n)}$ has at most $2^{n+1}$ elements, and hence for every fixed $\delta>0$, for all $n \leq (1-\delta)\log_2(p)$,  $\mu_a^{(n)}$ is far from uniform in total variation. However, if $\mu$ is uniform on $\{-1,0,1\}$, as above, then $H(\mu)=\log 3>\log 2$. This means that the value $a=2$ is exceptional (i.e. belongs to the $\eps$-fraction of bad multipliers in Theorem \ref{TVcutoff} and Corollary \ref{cor4}) for all primes $p$. This  happens in general whenever the residue $a \in \F_p$ has a representative in $\Z$ of size $<\exp H(\mu)$  (or more generally if $a$ is the reduction modulo a prime ideal of norm $p$ of a fixed algebraic number with Mahler measure $<\exp H(\mu)$). On the other hand, it is well-known and easy to check that apart from a density zero family of primes, any given $a \in \Z\setminus\{-1,0,1\}$ has multiplicative order at least $c \sqrt{p}/\log p$, and it is also known modulo GRH that given any function $\eps(p)>0$ tending to $0$ when $p \to +\infty$, $a$ has multiplicative order even at least $\eps(p)p$ for a density one set of primes, see \cite{erdos-murty}. 

\subsection{Diameter bounds} The support of the distribution is the following subset of $\F_p$ $$\Supp(\mu_a^{(n)})=\{b_0+b_1a +\ldots+b_{n-1}a^{n-1}, b_i \in \Supp(\mu)\} \subset \F_p.$$ It is of arithmetical interest to determine an upper bound on the diameter $D_a(p)$, i.e. the first $n$, for which the support is all of $\F_p$. 
Clearly $D_a(p)$ is bounded above by the uniform mixing time $T_\infty(\delta)$ for any $\delta<1$. But as recalled in $(\ref{tq})$, $T_\infty(\delta) \leq 2T_2(\delta^2)$, which is a general fact about Markov chains. In particular, the following is true.
\begin{proposition}For all primes $p$ and all $a \in \F_p^\times$ we have $$\Supp(\mu_a^{(n)})=\F_p$$
for all $n > 2 T_2(1)$. In other words $D_a(p) \leq 2T_2(1)+1$. 
\end{proposition}

We may thus apply the upper bounds on $T_2(\delta)$ given by Theorems \ref{loglog}, \ref{uncond}, \ref{konthm} or \ref{ell2cutoff} to deduce corresponding diameter bounds in each setting.

\subsection{Connections with Lehmer's problem} Recall that the Mahler measure of polynomial $P \in \Z[x]$ of degree $n$ and dominant term $a_n$ is the real number $$M(P)=|a_n| \prod_1^n \max\{1,|\alpha_i|\}$$ where the $\alpha_i$'s are the complex roots of $P$. The celebrated Lehmer problem \cite{lehmer} asks whether there is an absolute constant  $\eps_0>0$ such that $$M(P)>1+\eps_0$$ for every polynomial $P$ that is not a product of cyclotomic or monomial factors.

In \cite{breuillard-varju-Lehmer}, we gave an equivalent formulation of the Lehmer problem,  and proved that it is equivalent to the existence of a constant $\delta_0>0$ such that for a density one set of primes $p$, for every $a \in \F_p^\times$ of multiplicative order at least $(\log p)^2$ we have \begin{equation}\label{diam-leh}|\Supp(\mu_a^{(\lfloor \log p \rfloor)})| \ge p^{\delta_0}.\end{equation}
In particular, we see that if Lehmer's problem has a negative answer, then there is a function $B(p)$ with $B(p)\to +\infty$  as $p\to +\infty$ such that for a density one set of primes, there will be residues $a \in \F_p^\times$ with  multiplicative order at least $(\log p)^2$, (we can in fact go up to $p^{1/2-o(1)}$ and even up to any $o(p)$ assuming GRH), for which the diameter, and hence the mixing times,  are at least $B(p) \log p$. This justifies the claim made in the introduction that an $O(\log p)$ mixing time for almost all primes and all residues of large enough multiplicative order would imply the Lehmer conjecture.

We do not know about the converse. Konyagin's universal upper bound from Theorem \ref{konthm} (whose proof uses Dobrowolski's bound towards the Lehmer problem) does not seem to improve significantly (at least not its quadratic nature in $\log p$) assuming a positive answer to Lehmer. With such an assumption, the best we can do at the moment is the bound $(\ref{diam-leh})$ or a similar lower bound on $H(\mu_a^{(\log p)})$, but this is not quite enough to bound from above the mixing time.

\subsection{The associated reversible chain} The chain $(\ref{chain})$ is not reversible. It is natural to ask what happens to the symmetrized chain, whose transition probabilities are $$q_a(x,y):=\frac{1}{2}(p_a(x,y)+p_a(y,x)),$$
where  $x,y \in \F_p$  and  $p_a(x,y)=\mu(\{b \in \Z, y=ax+b \mod p\})$ is the transition probability of the chain $(\ref{chain})$. It would be interesting to perform for the symmetrized chain $q_a$ a similar analysis as the one we did for $p_a$ in this paper.

It is clear, however, that $q_a$ is of a different nature compared to $p_a$, as it will typically take much longer to mix. To see this, observe that  it is a quotient of the simple random walk on the wreath product $\Z \wr \Z$, and its distribution at time $n$ is $\nu_a^{(n)}=\pi_a(\nu^{(n)})$, where $\pi_a:\Z[x,x^{-1}] \to \F_p$ is the evaluation map $P \mapsto P(a)$ and $\nu^{(n)}$ is the associated law on the Laurent polynomials $\Z[x,x^{-1}]$. In particular, $$p\|\nu_a^{(n)}-u\|_2^2 = p\|\nu_a^{(n)}\|_2^2 -1 \ge p\|\nu^{(n)}\|_2^2-1.$$ But $\|\nu^{(n)}\|_2^2$ can be bounded below by the return probability of a certain symmetric random walk on $\Z\wr\Z$, which is of order $e^{-n^{1/3}(\log n)^{2/3}}$ (\cite{erschler-ptrf}*{Thm 2}). It follows that the $\ell^2$ mixing time of the symmetrized chain is at least $(\log p)^{3-o(1)}$. 

\subsection*{Acknowledgements} It is a pleasure to thank Boris Bukh, Persi Diaconis, Jonathan Hermon and Ariel Rapaport for helpful conversations. The first named author is grateful to the Newton Institute in Cambridge for its hospitality while part of this work was conducted.  We would also like to thank the referee for their very careful examination of our manuscript.

\bibliographystyle{abbrv}
\bibliography{bibfile}

\begin{bibdiv}
\begin{biblist}

\bib{breuillard-varju0}{article}{
      author={Breuillard, E.},
      author={Varj\'u, P.~P.},
       title={Entropy of bernoulli convolutions and uniform exponential growth
  for linear groups},
        date={2018},
        note={arXiv:1510.04043v3, to appear in J. d'Analyse Math.},
}

\bib{breuillard-varju-irred}{article}{
      author={Breuillard, E.},
      author={Varj\'u, P.~P.},
       title={Irreducibility of random polynomials of large degree},
        date={2018},
        note={arXiv:1810.13360},
}

\bib{breuillard-varju-Lehmer}{article}{
      author={Breuillard, E.},
      author={Varj\'u, P.~P.},
       title={On the {L}ehmer conjecture and counting in finite fields},
        date={2019},
     journal={Discrete Analysis},
      volume={5},
}

\bib{CDG-RW}{article}{
      author={Chung, F. R.~K.},
      author={Diaconis, Persi},
      author={Graham, R.~L.},
       title={Random walks arising in random number generation},
        date={1987},
        ISSN={0091-1798},
     journal={Ann. Probab.},
      volume={15},
      number={3},
       pages={1148\ndash 1165},
  url={http://links.jstor.org/sici?sici=0091-1798(198707)15:3<1148:RWAIRN>2.0.CO;2-6&origin=MSN},
      review={\MR{893921}},
}

\bib{Coh-computational-NT}{book}{
      author={Cohen, Henri},
       title={A course in computational algebraic number theory},
      series={Graduate Texts in Mathematics},
   publisher={Springer-Verlag, Berlin},
        date={1993},
      volume={138},
        ISBN={3-540-55640-0},
         url={https://doi.org/10.1007/978-3-662-02945-9},
      review={\MR{1228206}},
}

\bib{diaconis-survey}{article}{
      author={Diaconis, Persi},
       title={The cutoff phenomenon in finite {M}arkov chains},
        date={1996},
        ISSN={0027-8424},
     journal={Proc. Nat. Acad. Sci. U.S.A.},
      volume={93},
      number={4},
       pages={1659\ndash 1664},
         url={https://doi.org/10.1073/pnas.93.4.1659},
      review={\MR{1374011}},
}

\bib{Dob-Lehmer}{article}{
      author={Dobrowolski, E.},
       title={On a question of {L}ehmer and the number of irreducible factors
  of a polynomial},
        date={1979},
        ISSN={0065-1036},
     journal={Acta Arith.},
      volume={34},
      number={4},
       pages={391\ndash 401},
         url={https://doi.org/10.4064/aa-34-4-391-401},
      review={\MR{543210}},
}

\bib{erdos-murty}{incollection}{
      author={Erd\H{o}s, P\'{a}l},
      author={Murty, M.~Ram},
       title={On the order of {$a\pmod p$}},
        date={1999},
   booktitle={Number theory ({O}ttawa, {ON}, 1996)},
      series={CRM Proc. Lecture Notes},
      volume={19},
   publisher={Amer. Math. Soc., Providence, RI},
       pages={87\ndash 97},
      review={\MR{1684594}},
}

\bib{erschler-ptrf}{article}{
      author={Erschler, Anna},
       title={Isoperimetry for wreath products of {M}arkov chains and
  multiplicity of selfintersections of random walks},
        date={2006},
        ISSN={0178-8051},
     journal={Probab. Theory Related Fields},
      volume={136},
      number={4},
       pages={560\ndash 586},
         url={https://doi.org/10.1007/s00440-005-0495-7},
      review={\MR{2257136}},
}

\bib{Hel-Arizona}{article}{
      author={Helfgott, Harald~Andres},
       title={Growth and expansion in algebraic groups over finite fields},
        date={2019},
        note={arXiv:1902.06308v3},
}

\bib{Hildebrand-AnnalsOfProba}{article}{
      author={Hildebrand, Martin},
       title={Random processes of the form {$X_{n+1}=a_nX_n+b_n\pmod p$}},
        date={1993},
        ISSN={0091-1798},
     journal={Ann. Probab.},
      volume={21},
      number={2},
       pages={710\ndash 720},
  url={http://links.jstor.org/sici?sici=0091-1798(199304)21:2<710:RPOTF>2.0.CO;2-I&origin=MSN},
      review={\MR{1217562}},
}

\bib{hildebrand-on-the-CDG-process}{article}{
      author={Hildebrand, Martin},
       title={On the {C}hung-{D}iaconis-{G}raham random process},
        date={2006},
        ISSN={1083-589X},
     journal={Electron. Comm. Probab.},
      volume={11},
       pages={347\ndash 356},
         url={https://doi.org/10.1214/ECP.v11-1237},
      review={\MR{2274529}},
}

\bib{hildebrand-lower-bound}{article}{
      author={Hildebrand, Martin},
       title={On a lower bound for the {C}hung-{D}iaconis-{G}raham random
  process},
        date={2019},
        ISSN={0167-7152},
     journal={Statist. Probab. Lett.},
      volume={152},
       pages={121\ndash 125},
         url={https://doi.org/10.1016/j.spl.2019.04.020},
      review={\MR{3953053}},
}

\bib{Hil-thesis}{book}{
      author={Hildebrand, Martin~Victor},
       title={Rates of convergence of some random processes on finite groups},
   publisher={ProQuest LLC, Ann Arbor, MI},
        date={1990},
  url={http://gateway.proquest.com/openurl?url_ver=Z39.88-2004&rft_val_fmt=info:ofi/fmt:kev:mtx:dissertation&res_dat=xri:pqdiss&rft_dat=xri:pqdiss:9035480},
        note={Thesis (Ph.D.)--Harvard University},
      review={\MR{2638852}},
}

\bib{Kon-RW}{article}{
      author={Konyagin, S.~V.},
       title={Estimates for {G}aussian sums and {W}aring's problem modulo a
  prime},
        date={1992},
        ISSN={0371-9685},
     journal={Trudy Mat. Inst. Steklov.},
      volume={198},
       pages={111\ndash 124},
      review={\MR{1289921}},
}

\bib{lehmer}{article}{
      author={Lehmer, D.~H.},
       title={Factorization of certain cyclotomic functions},
        date={1933},
        ISSN={0003-486X},
     journal={Ann. of Math. (2)},
      volume={34},
      number={3},
       pages={461\ndash 479},
         url={https://doi.org/10.2307/1968172},
      review={\MR{1503118}},
}

\bib{lehmer-comp}{inproceedings}{
      author={Lehmer, D.~H.},
       title={Mathematical methods in large-scale computing units},
        date={1951},
   booktitle={Proceedings of a {S}econd {S}ymposium on {L}arge-{S}cale
  {D}igital {C}alculating {M}achinery, 1949},
   publisher={Harvard University Press, Cambridge, Mass.},
       pages={141\ndash 146},
      review={\MR{0044899}},
}

\bib{peres-levin}{book}{
      author={Levin, David~A.},
      author={Peres, Yuval},
       title={Markov chains and mixing times},
   publisher={American Mathematical Society, Providence, RI},
        date={2017},
        ISBN={978-1-4704-2962-1},
        note={Second edition},
      review={\MR{3726904}},
}

\bib{Mah-discriminant}{article}{
      author={Mahler, K.},
       title={An inequality for the discriminant of a polynomial},
        date={1964},
        ISSN={0026-2285},
     journal={Michigan Math. J.},
      volume={11},
       pages={257\ndash 262},
         url={http://projecteuclid.org/euclid.mmj/1028999140},
      review={\MR{0166188}},
}

\bib{Sta-Dedekind}{article}{
      author={Stark, H.~M.},
       title={Some effective cases of the {B}rauer-{S}iegel theorem},
        date={1974},
        ISSN={0020-9910},
     journal={Invent. Math.},
      volume={23},
       pages={135\ndash 152},
         url={https://doi.org/10.1007/BF01405166},
      review={\MR{0342472}},
}

\end{biblist}
\end{bibdiv}

\end{document}